\documentclass[11pt]{article}

\usepackage[T1]{fontenc}
\usepackage[utf8]{inputenc}
\usepackage[a4paper,margin=29mm]{geometry}
\usepackage{lmodern}
\usepackage{microtype}
\usepackage{amsmath,amssymb,amsthm,mathtools}
\usepackage{array,booktabs,enumitem}
\usepackage{xcolor}
\usepackage[colorlinks=true,linkcolor=blue!55!black,citecolor=blue!55!black,urlcolor=blue!55!black]{hyperref}

\setlist[itemize]{leftmargin=2em,itemsep=2pt,topsep=4pt}
\setlist[enumerate]{leftmargin=2.2em,itemsep=2pt,topsep=4pt}
\allowdisplaybreaks

\newtheorem{definition}{Definition}[section]
\newtheorem{proposition}[definition]{Proposition}
\newtheorem{lemma}[definition]{Lemma}
\newtheorem{theorem}[definition]{Theorem}
\newtheorem{corollary}[definition]{Corollary}

\newtheorem{remark}[definition]{Remark}

\newcommand{\EnK}{E_{\mathrm{nK}}}
\newcommand{\EnKtilde}{\widetilde E_{\mathrm{nK}}}
\newcommand{\EnKas}{E_{\mathrm{nK}}^{\mathrm{as}}}
\newcommand{\Eplus}{E_{+}}
\newcommand{\ddc}{dd^c}

\newcommand{\C}{\mathbb{C}}

\newcommand{\K}{\mathcal{K}}
\newcommand{\Pcone}{\mathcal{P}}

\DeclareMathOperator{\Null}{Null}
\DeclareMathOperator{\Sing}{Sing}
\DeclareMathOperator{\Reg}{Reg}
\DeclareMathOperator{\supp}{supp}

\newcommand{\dc}{d^c}

\title{Big and nef cohomology classes on compact K\"ahler spaces}
\author{Duc-Viet Vu}
\date{}

\begin{document}
\maketitle
\vspace{-1.2em}

\begin{abstract}
We first prove a version of the Demailly--P{\u{a}}un theorem for compact
weakly K\"ahler spaces.  Secondly, for a big and nef Bott--Chern class on a
compact normal K\"ahler space, we prove that the
restricted non--K\"ahler locus, defined using
K\"ahler currents with weak analytic singularities, is equal to the usual
non--K\"ahler locus.
\end{abstract}

\section{Introduction}

The aim of this work is to study big and nef cohomology classes on compact K\"ahler spaces. For smooth K\"ahler manifolds, Demailly--P{\u{a}}un \cite{DP04} proved a transcendental version of the Nakai--Moishezon characterization of the K\"ahler cone.  Our first goal is to establish a complete generalization of this fundamental result to compact K\"ahler spaces. There are already several versions of this theorem for singular spaces.

Collins--Tosatti proved a singular Demailly--P{\u{a}}un theorem for a compact analytic subvariety of a smooth K\"ahler manifold, with $(1,1)$-classes obtained from the ambient manifold \cite[Theorem~1.1]{CT16Singular}.  In their paper, they explicitly pointed to the extension of the Demailly--P{\u{a}}un theorem to general compact K\"ahler (reduced) and irreducible analytic spaces as a problem for future work
\cite[Introduction]{CT16Singular}. Das--Hacon--P{\u{a}}un  \cite[Theorem~2.36]{DHP24} generalized main results in \cite{DP04} to  normal compact K\"ahler spaces; see also \cite[Theorem~6.9]{BL22}. We refer to \cite{Buc99,Buc00,Chi16,Dan26,Pop16,Lam99} for further work for non-K\"ahler manifolds or varieties.
It should be noted that in the algebraic setting, the classical
algebraic statement does not require smoothness: Nakai treated
projective schemes, Moishezon treated complete algebraic varieties, and
Kleiman's argument applies to arbitrary proper schemes
\cite{Nak63,Moi64,Kle66}; see also the exposition \cite{Elk74}.  Thus the integral Nakai--Moishezon criterion was already known for singular, and even
nonreduced, algebraic schemes.   For real divisors on projective varieties, an
earlier extension was obtained by Campana--Peternell \cite{CP90}; see also
Lazarsfeld's treatment \cite[Theorem~2.3.18]{Laz04}.  Fujino--Miyamoto later
proved the Nakai--Moishezon criterion for real line bundles on arbitrary
complete schemes and complete algebraic spaces
\cite[Theorems~1.3 and~1.6]{FM23}.  For the
integral criterion on complete algebraic spaces, see also Koll\'ar and
Pascual--Gainza \cite{Kol90,PG85}. In order to state our first main result, we need to introduce some necessary notions. 

Let $X$ be a complex space 
and $C^\infty_{1,1}(X,\mathbb R)$ be the vector space of real smooth $(1,1)$-forms on $X$. In Section \ref{sec-premi},  we will define \emph{the extended Bott-Chern group $ \widetilde H_{\mathrm{BC}}^{1,1}(X,\mathbb R)$ of $X$}, which, unlike the usual Bott-Chern group $ H_{\mathrm{BC}}^{1,1}(X,\mathbb R)$, uses also closed forms which do not necessarily have local potentials. 
For  every
reduced complex space $X$, there is a natural
injective map
\begin{equation}
\label{eq:BC-into-extended}
    H_{\mathrm{BC}}^{1,1}(X,\mathbb R)
    \lhook\joinrel\longrightarrow
    \widetilde H_{\mathrm{BC}}^{1,1}(X,\mathbb R).
\end{equation}
The equality occurs if \(X\) is a complex manifold. 

Let $\widetilde{\K}$ denote the cone of classes in $\widetilde H_{\mathrm{BC}}^{1,1}(X,\mathbb R)$, which contain weakly K\"ahler forms, and let
$\widetilde{\Pcone}$ be the cone of classes having positive top
self-intersection on every positive-dimensional irreducible analytic subset. Let $\K$ be the usual Kähler cone in the Bott-Chern cohomology $ H_{\mathrm{BC}}^{1,1}(X,\mathbb R)$. Observe that 
$$\K \subset \widetilde{\K} \cap  H_{\mathrm{BC}}^{1,1}(X,\mathbb R).$$
We refer to Section \ref{sec-premi} for precise meanings of these notions as well as those used in Theorem \ref{thm:main-intro} below. Here is our first main result. 

\begin{theorem} \label{thm:main-intro}
Let $X$ be a compact weakly K\"ahler space. 
Then the cone
$\widetilde{\K}$ is a connected component of
$\widetilde{\Pcone}$.  If, moreover, $X$ is K\"ahler, then inside
Bott--Chern cohomology the K\"ahler cone $\K$ is a connected component of
the corresponding numerical positive cone $\Pcone$.
\end{theorem}

In particular, no normality or global ambient embedding is required for $X$. The proof of Theorem \ref{thm:main-intro} is based essentially on arguments from \cite{DP04,Pau98}.

Consider now $X$ of pure dimension. We define now another version of positive cones. Let $\widetilde{\K}'$ be the cone of classes in $\widetilde H_{\mathrm{BC}}^{1,1}(X,\mathbb R)$, which contain weakly Kähler forms having local $\ddc$-potentials.  Let $\widetilde{\K}''$ be the cone of classes in $\widetilde H_{\mathrm{BC}}^{1,1}(X,\mathbb R)$, which contain smooth Kähler $(1,1)$-currents on $X$. One has
$$\widetilde{\K}' \subset \widetilde{\K} \subset \widetilde{\K}''.$$
 Observe that the cone of nef classes in $\widetilde H_{\mathrm{BC}}^{1,1}(X,\mathbb R)$ is equal to the closure $\overline{\widetilde{\K}''}$ of $\widetilde{\K}''$ in $\widetilde H_{\mathrm{BC}}^{1,1}(X,\mathbb R)$ (see Remark \ref{re-subtlety} below for more information).

As a consequence of Theorem \ref{thm:main-intro}, we obtain the following equality between different ``Kähler'' cones.

\begin{corollary} \label{cor-main-equalitycone} Let $X$ be a compact weakly K\"ahler space  of pure dimension. Then we have
\begin{align}\label{eq-conebangnhau}
\widetilde{\K} =\widetilde{\K}''.
\end{align}
In particular, nef cohomology classes in  $\widetilde H_{\mathrm{BC}}^{1,1}(X,\mathbb R)$ can be approximated by weakly Kähler classes. Moreover, one has 
\begin{align} \label{eq-conebangnhauBC}
\K = \widetilde{\K}' \cap     H_{\mathrm{BC}}^{1,1}(X,\mathbb R)= \widetilde{\K}'' \cap     H_{\mathrm{BC}}^{1,1}(X,\mathbb R) .
\end{align}
In particular,  if $X$ is Kähler, then every nef cohomology class in  $H_{\mathrm{BC}}^{1,1}(X,\mathbb R)$ can be approximated by Kähler classes.
\end{corollary}
One should note that in the smooth setting, Corollary \ref{cor-main-equalitycone} is trivial. By contrast, in the singular setting, we have to use Theorem \ref{thm:main-intro} to obtain it. We don't have an easier way to prove Corollary \ref{cor-main-equalitycone}.

As in \cite{DP04}, we obtain the following direct consequence of Theorem \ref{thm:main-intro} and Corollary \ref{cor-main-equalitycone}.

\begin{theorem} \label{thm-projective} Let $X$ be a projective variety of pure dimension. Then we have 
$$\widetilde{\K}= \widetilde{\Pcone},\quad  \K = \Pcone.$$
Furthermore, the following assertions hold:

(i) a class $\alpha \in H_{\mathrm{BC}}^{1,1}(X,\mathbb R)$ is Kähler if and only if $\int_Y \alpha^p>0$ for every positive-dimensional irreducible $Y\subset X$ ($p=\dim Y$).  

(ii) a class $\alpha \in H_{\mathrm{BC}}^{1,1}(X,\mathbb R)$ is nef if and only if $\int_Y \alpha^p \ge 0$ for every positive-dimensional irreducible $Y\subset X$ ($p=\dim Y$).     
\end{theorem}

For a big and nef Bott--Chern class $\alpha$ on a  compact K\"ahler
space of pure dimension, denote by $\EnK(\alpha)$  the intersection of the
positive Lelong loci of K\"ahler currents in $\alpha$ which have local
potentials. Let $\EnKas(\alpha)$ denote the same intersection restricted to K\"ahler
currents with weak analytic singularities, and let $\Null(\alpha)$ be the null
locus of $\alpha$. We refer to Section \ref{sec-nonkahlerlocus} for precise definitions. When $X$ is smooth, then by Demailly's analytic approximation of psh functions, we get
\begin{align}\label{eq-EasEnonkahlerlocus}
\EnK(\alpha)=\EnKas(\alpha).
\end{align}
For $X$ smooth, Collins--Tosatti \cite{CT15} proved
\begin{align}\label{eq-nulllocus}
\EnKas(\alpha)=\Null(\alpha)
\end{align}
This extends previous results for integral big and nef classes
\cite{Nak00,ELMNP09}.
For integral classes, the algebraic counterpart of (\ref{eq-nulllocus}) takes the form
$$\mathbf B_+(L)=\operatorname{Null}(L),$$
where $L$ is a big and nef line bundle on a smooth projective manifold. An extension of the last equality to log-canonical varieties was obtained in \cite{CL14}. The general result for nef $\mathbb R$-Cartier divisors on arbitrary projective schemes is due to \cite{Bir17}. We refer to  \cite{BCL14,Lop15} for related information.

Recently, Hacon--P{\u{a}}un established (\ref{eq-nulllocus})
for normal compact K\"ahler spaces and expected (\ref{eq-EasEnonkahlerlocus}) holds in this setting, see \cite[Theorem~4.21 and Remark~4.19]{HPaun24}. Our second main result (Theorem \ref{thm:restricted-non-kahler-intro}) confirms this expectation.

\begin{theorem}\label{thm:restricted-non-kahler-intro}
Let $X$ be a compact normal K\"ahler space, and let
$\alpha\in H_{\mathrm{BC}}^{1,1}(X,\mathbb R)$ be big and nef.  Then
\[
\EnK(\alpha) =\EnKas(\alpha)=\Null(\alpha).
\]
\end{theorem}

Basic facts about psh functions on complex spaces will be recalled in Section \ref{sec-premi}. Our first main result and its consequence are proved in Section \ref{sec-Kählercone}. Theorem \ref{thm:restricted-non-kahler-intro} will be proved in Section \ref{sec-nonkahlerlocus}.

\medskip
\noindent\textbf{Acknowledgement.} We thank Duc-Bao Nguyen for many fruitful
discussions. The research of D.-V. Vu is partially funded by the DFG-Project
VU 126/1-2.

\section{Preliminaries} \label{sec-premi}

Let $X$ be a (reduced) complex space and $C^\infty_{1,1}(X,\mathbb R)$ be the vector space of real smooth $(1,1)$-forms on $X$. We refer to \cite{Dem85} for the definitions of smooth forms, psh functions, and their basic properties. 
We define
\begin{equation}
\label{eq:extended-ddbar-group}
    \widetilde H_{\mathrm{BC}}^{1,1}(X,\mathbb R)
    :=
    \frac{
        \left\{
        \alpha\in C^\infty_{1,1}(X,\mathbb R)
        \;\middle|\;
        d\alpha=0,
        \right\}
    }{
        \sqrt{-1}\partial\bar\partial
        C^\infty(X,\mathbb R)
    }.
\end{equation}
Here the tilde is introduced to distinguish this quotient from the
Bott--Chern group defined using forms with local potentials. The advantage of $\widetilde H_{\mathrm{BC}}^{1,1}(X,\mathbb R)$ is that it allows us to treat smooth closed forms that do not necessarily have local $\ddc$-potentials.  Thus every smooth \(d\)-closed real \((1,1)\)-form determines a class
\[
    \{\alpha\}
    \in
    \widetilde H_{\mathrm{BC}}^{1,1}(X,\mathbb R),
\]
whether or not \(\alpha\) has smooth local
\(\sqrt{-1}\partial\bar\partial\)-potentials.

We endow the numerator of \eqref{eq:extended-ddbar-group} with its natural
Fr\'echet topology and the quotient with the quotient topology. It seems that  this topology need not be Hausdorff in general.  All closures and connected components
below are taken in this topology.  The same convention is used for the
Bott--Chern group.  We recall that the Bott--Chern group of $X$ is defined as follows:
\begin{equation}
\label{eq:BC-local-potentials}
    H_{\mathrm{BC}}^{1,1}(X,\mathbb R)
    =
    \frac{
        \left\{
        \alpha\in C^\infty_{1,1}(X,\mathbb R)
        \;\middle|\;
        \alpha \text{ has smooth local }
        \sqrt{-1}\partial\bar\partial\text{-potentials}
        \right\}
    }{
        \sqrt{-1}\partial\bar\partial
        C^\infty(X,\mathbb R)
    }.
\end{equation}
A form having local potentials is automatically closed. Hence, for every
reduced complex space $X$, there is a natural
injective map
\begin{equation*}
    H_{\mathrm{BC}}^{1,1}(X,\mathbb R)
    \lhook\joinrel\longrightarrow
    \widetilde H_{\mathrm{BC}}^{1,1}(X,\mathbb R).
\end{equation*}
The equality occurs whenever \(X\) is a complex manifold. On a singular normal K\"ahler space, we do not know if both groups are the same. One only has
\[
    \left\{
    \text{smooth \((1,1)\)-forms with local potentials}
    \right\}
    \subseteq
    \left\{
    \text{smooth \(d\)-closed \((1,1)\)-forms}
    \right\}.
\]
The inclusion is likely to be strict in general.  We refer to \cite{Boucksom-Guedj-regularKEflow} for a treatment on the Bott-Chern cohomology for normal varieties.


A \emph{Hermitian metric} $\omega$ on $X$ is a real smooth $(1,1)$-form on $X$ such
that for every point $x$ of $X$, there exists a local embedding $\rho: U_x\hookrightarrow W \subset \C^N$ of an open neighborhood $U_x$ of $x$ into an open subset $W$ of $\C^N$ for some $N>0$ such that $\omega|_{U_x}$ is the restriction of a strictly positive smooth $(1,1)$-form on $W$.  Closedness is not part
of this definition.  We say that $\omega$ is \emph{weakly K\"ahler} if it is
Hermitian and $d$-closed, and that it is \emph{K\"ahler} if it is the restriction  under some local embedding to $\C^N$
of a strictly positive closed smooth $(1,1)$-form. 
The space $X$ is weakly K\"ahler, respectively K\"ahler, if it possesses a
form of the corresponding kind.  Every K\"ahler form is weakly K\"ahler. The notion of Kähler forms for singular spaces was introduced in \cite[Page 346]{Grauert-Modifikationen}.  We refer to  \cite{Var89,Graf-Kirschner,Pau98} for more information on K\"ahler spaces. Note that our notion of weakly K\"ahler metric is different from that in \cite{Var89}. The following result was proved in  \cite[Lemma 4]{Narashimhan-62} or \cite[Page 335]{Grauert-Modifikationen}. We include a proof for the readers' convenience.

\begin{lemma} \label{le-kählerkophuthuocvaolocal} Let $X$ be a complex space.  Let $\omega$ be a Hermitian form on $X$. Then for every $x \in X$ and every local embedding $\rho:U_x \to W \subset \C^N$ of $X$, where $U_x$ is an open neighborhood of $x$ in $X$, there exist an open neighborhood $W_x$ of $\rho(x)$ in $W$ and a Hermitian $(1,1)$-form $\eta$ on $W_x$ such that $$\omega= \rho^*\eta$$ 
on $\rho^{-1}(W_x)$. Furthermore, if $\omega$ is K\"ahler, then we can choose $\eta$ to be K\"ahler.    
\end{lemma}

\begin{proof}
Fix \(x\in X\) and a local embedding
\[
\rho\colon U_x\hookrightarrow W\subset \mathbb C^N.
\]
By the definition of a Hermitian form, after shrinking \(U_x\) there exist
another local embedding
\[
\sigma\colon U_x\hookrightarrow \Omega\subset \mathbb C^M
\]
and a strictly positive smooth real \((1,1)\)-form
\(\widetilde\omega\) on \(\Omega\) such that
\[
\omega|_{U_x}=\sigma^*\widetilde\omega.
\]
Put \(y:=\rho(x)\). The biholomorphism
\[
\sigma\circ\rho^{-1}\colon \rho(U_x)\longrightarrow \sigma(U_x)
\]
and its inverse are holomorphic maps between analytic subspaces. Hence,
after shrinking the ambient neighborhoods, they extend to holomorphic maps
\[
F\colon (W,y)\longrightarrow (\Omega,\sigma(x)),
\qquad
G\colon (\Omega,\sigma(x))\longrightarrow (W,y),
\]
such that
\[
F\circ\rho=\sigma,
\qquad
G\circ\sigma=\rho
\]
on \(U_x\).
Define the holomorphic map
\[
h:=\operatorname{id}_W-G\circ F\colon W\longrightarrow \mathbb C^N.
\]
Since \(G\circ F\) restricts to the identity on \(\rho(U_x)\), we have
\[
h|_{\rho(U_x)}=0.
\]
Furthermore, if \(v\in T_yW\) satisfies \(dF_y(v)=0\), then
\[
dh_y(v)
 =v-dG_{\sigma(x)}\bigl(dF_y(v)\bigr)
 =v.
\]
It follows that
\[
(dF_y,dh_y)\colon T_yW
   \longrightarrow T_{\sigma(x)}\Omega\oplus\mathbb C^N
\]
is injective.

Consider the smooth real \((1,1)\)-form
\[
\eta
 :=
 F^*\widetilde\omega
 +\sqrt{-1}\,\partial\bar\partial\lVert h\rVert^2
\]
on a sufficiently small neighborhood of \(y\). Since \(h\) is holomorphic,
\[
\sqrt{-1}\,\partial\bar\partial\lVert h\rVert^2
 =
 \sqrt{-1}\sum_{j=1}^N dh_j\wedge d\overline{h_j}
\]
is semipositive. The injectivity of \((dF_y,dh_y)\), together with the
strict positivity of \(\widetilde\omega\), shows that \(\eta\) is strictly
positive at \(y\). After shrinking to an open neighborhood
\(W_x\Subset W\) of \(y\), the form \(\eta\) is therefore Hermitian on
\(W_x\).

Finally, since \(h\circ\rho=0\), we obtain
\[
\begin{aligned}
\rho^*\eta
&=\rho^*F^*\widetilde\omega
  +\sqrt{-1}\,\partial\bar\partial
       \lVert h\circ\rho\rVert^2 \\
&=\sigma^*\widetilde\omega
 =\omega
\end{aligned}
\]
on \(\rho^{-1}(W_x)\).

If \(\omega\) is Kähler, then in the initial choice one may take
\(\widetilde\omega\) to be \(d\)-closed. Both
\(F^*\widetilde\omega\) and
\(\sqrt{-1}\,\partial\bar\partial\lVert h\rVert^2\) are then \(d\)-closed.
Consequently, \(\eta\) is a Kähler form on \(W_x\).
\end{proof}

We deduce from the last result the following standard fact. 

\begin{lemma} \label{le-convexKähler} Let $X$ be a complex space. Then the following assertions hold:

(i) If $\omega_1, \omega_2$ are weakly Kähler forms, then $\omega_1+\omega_2$ is so,

(ii) If $\omega_1, \omega_2$ are Kähler forms, then $\omega_1+\omega_2$ is so.    
\end{lemma}

\begin{proof}
For (i), Lemma~\ref{le-kählerkophuthuocvaolocal} allows us, for every
$x\in X$, to use one fixed local embedding and to extend both forms to
Hermitian forms on the ambient neighborhood.  Their sum is Hermitian, and
$d(\omega_1+\omega_2)=0$; hence it is weakly K\"ahler.  For (ii), the same
lemma gives closed Hermitian ambient extensions, whose sum is K\"ahler.
\end{proof}

For the remainder of this section, when discussing currents of a
fixed bidegree, we assume that \(X\) is of pure dimension \(n\). Let
$\mathcal D^{a,b}(X)$ denote compactly supported smooth $(a,b)$-forms on $X$ with
their usual topology.  A current of \emph{bidimension}
$(q,r)$ is a continuous linear functional on $\mathcal D^{q,r}(X)$.  The same
current is said to have \emph{bidegree} $(n-q,n-r)$. 
The differential of a current is defined by duality, with the usual sign.
A current is \emph{closed} if $dT=0$; for a current of pure type this is
equivalent to $\partial T=\bar\partial T=0$.

A smooth $(k,k)$-form is \emph{(strongly) positive} if it is locally the restriction of (strongly) positive forms on ambient spaces under local embeddings.   
A current $T$ of bidegree
$(1,1)$ on $X$ is \emph{positive}, written $T\geq0$,
if
\[
  \langle T,\Phi\rangle\geq 0
\]
for every compactly supported positive $(n-1,n-1)$-form $\Phi$.
Equivalently, for every local embedding $\jmath:U\hookrightarrow\Omega$ of $X$, the current $\jmath_*T$ is a (weakly) positive
current on $\Omega$.

A function $u$ on $X$ is said to be \emph{psh} (resp. \emph{quasi-psh}) if for every point $x$ of $X$, there exists a local embedding $\rho: U_x \hookrightarrow W \subset \C^N$ of an open neighborhood $U_x$ of $x$ into an open subset $W$ of $\C^N$ for some $N>0$ such that $u$ is the restriction of a psh (resp. quasi-psh) function on $W$. Let $\theta$ be a smooth $(1,1)$-form on $X$. We say that $u$ is \emph{$\theta$-psh}  if $u$ is quasi-psh and we have $\ddc u+ \theta \ge 0$ as currents on $X$  (recall that $d^c:=\frac{i}{2 \pi} (\bar \partial -\partial)$). A closed positive $(1,1)$-current $T$ on $X$ is said to have \emph{local potentials} if locally $T= \ddc u$ for some local psh function $u$ on $X$.

We recall the following definitions; see \cite[Definition~3.1]{DH24} and
\cite[Definition~2.6]{HLR26}.

\begin{definition}
Let $\alpha=[\theta]\in
\widetilde H_{\mathrm{BC}}^{1,1}(X,\mathbb R)$, where $\theta$ is a
smooth closed real $(1,1)$-form on $X$, and let $\omega$ be a Hermitian form
on $X$.

(i) Let $T$ be a closed positive $(1,1)$-current on $X$.  We say that $T$ is a current in
$\alpha$ and write $T \in \alpha$ if we have
\[
 T=\theta+\ddc\varphi\geq 0
\]
as currents on $X$,  for a $\theta$-psh function $\varphi$ on $X$;

(ii) A $(1,1)$-current $T$ on $X$ is \emph{strictly positive} if  there is a
number $\delta>0$ such that
\[
 T\geq\delta\omega
\]
as currents on $X$; $T$ is called  \emph{a K\"ahler current} if $T$ is closed and strictly positive. The class $\alpha$ is \emph{big} if it contains a
K\"ahler current.  Equivalently, there are
a $\theta$-psh function $\varphi$ and a number $\delta>0$ with
\[
 \theta+\ddc\varphi\geq\delta\omega.
\]

(iii)  $\alpha$ is \emph{nef}  if, for every $\varepsilon>0$, there exists
$f_\varepsilon\in\mathcal C^\infty(X,\mathbb R)$ such that
\[
 \theta+\ddc f_\varepsilon\geq-\varepsilon\omega.
\]
\end{definition}

We note that the above definitions are independent of the choice of smooth representative $\theta$ in $\alpha$. 
If $X$ is a compact weakly K\"ahler space, $\omega$ is a weakly K\"ahler form on $X$, and $\alpha$ is a nef class in $\widetilde H_{\mathrm{BC}}^{1,1}(X,\mathbb R)$, then the class $\alpha+ \epsilon \{\omega\}$ contains the smooth K\"ahler current
\[
   \theta+\epsilon\omega+\ddc f_{\epsilon/2}
      \geq \frac{\epsilon}{2}\omega.
\]
A priori, it is not clear if this form can be chosen to be weakly K\"ahler.  
By Corollary \ref{cor-main-equalitycone}, we know that there exists indeed such a form.  We conclude this section with the following remark.

\begin{remark} \label{re-subtlety} (i) In general, even when $X$ is normal, it seems that  a weakly K\"ahler form is not necessarily a K\"ahler form. However, we don't have an example.

(ii) 
In general, a smooth Kähler $(1,1)$-current having smooth local potentials is not necessarily a smooth Kähler form. We refer to \cite{DoVu-positivity} for a detailed discussion.   
\end{remark}

\section{Kähler cones} \label{sec-Kählercone}

We first recall the mass-concentration theorem of Demailly--P\u{a}un
\cite[Theorem~0.5]{DP04}.

\begin{theorem}\label{thm:current}
Let $M$ be a compact (connected) K\"ahler manifold of dimension $m$, and let $\beta$ be a
nef real $(1,1)$-class such that
\[
 \int_M\beta^m>0.
\]
Then $\beta$ contains a K\"ahler current
\[
 T=\beta_0+\ddc u\geq\varepsilon\omega_M,
\]
where $\beta_0$ is a smooth representative of $\beta$, $\omega_M$ is a
K\"ahler form, and $\varepsilon>0$.  Moreover, $u$ can be chosen smooth on
the complement of a proper analytic subset and with analytic, hence
logarithmic, singularities along that subset.
\end{theorem}

\begin{lemma}[{\cite[Lemma~5]{Dem90}}]
\label{le-Demailly-exten.lemma}
Let $Y$ be an analytic subvariety of a complex space $X$. Then, there exists a quasi-plurisubharmonic function $v$ on $X$ such that
\[
 v\in C^\infty(X\setminus Y),
 \qquad
 v=-\infty\quad\text{on }Y,
\]
and $v$ has logarithmic poles along $Y$.
\end{lemma}


We also need to use the following standard fact. We include a proof for the readers' convenience.

\begin{lemma}\label{le-blowup}
Let $X$ be a compact  weakly K\"ahler space, and let
$\pi\colon X'\to X$ be a desingularization obtained by compositions of successive blowups.  Then $X'$ is a compact
K\"ahler manifold.
\end{lemma}

\begin{proof}
Let $\omega_X$ be a weakly K\"ahler form on $X$.  Its pullback
$\pi^*\omega_X$ is a smooth closed semipositive $(1,1)$-form on $X'$.  At a
point $x'\in X'$, it is positive on every tangent vector which is not in
$\ker d\pi_{x'}$.

Since $\pi$ is projective, there is a $\pi$-relatively ample line bundle
$L$ on $X'$.  Choose a smooth Hermitian metric $h$ on $L$ whose curvature
form
\[
 \eta:=\frac{\sqrt{-1}}{2\pi}\Theta_h(L)
\]
is positive on the vertical tangent spaces $\ker d\pi$.  Equivalently, one
may embed $X'$ over $X$ into a relative projective space and restrict a
relative Fubini--Study form.  The form $\eta$ may have negative horizontal
eigenvalues, but these are uniformly bounded because $X'$ is compact.
Consequently, for all sufficiently large constants $C>0$, the closed form
\[
 \omega_{X'}:=C\pi^*\omega_X+\eta
\]
is positive definite.  Indeed, $\eta$ supplies positivity in the vertical
directions, while $C\pi^*\omega_X$ dominates the bounded negative part of
$\eta$ in the remaining directions.  Thus $\omega_{X'}$ is a K\"ahler form
on $X'$. 
\end{proof}

One can consult \cite{Var89} for more information on Kählerness under morphisms. 
The following result is a version of \cite[Theorem 4]{Dem90}, \cite[Proposition~3.3(i)]{DP04} and \cite[Lemma~1]{Pau98}.

\begin{lemma}\label{le-extension-DP}
Let $X$ be a complex space, let $A\subset X$ be a compact analytic subset,
and let $\theta$ be a smooth closed real $(1,1)$-form on $X$.  Assume that
there is a smooth function $\varphi_A$ on $A$ such that
\[
 \theta|_A+\ddc\varphi_A
\]
is weakly K\"ahler on $A$.  Then there exist an open neighborhood $U$ of
$A$ and a smooth function $\varphi$ on $U$ such that
\[
 \theta+\ddc\varphi
\]
is weakly K\"ahler on $U$. 
If $\theta$ is locally the restriction of a closed smooth form under local embeddings (or equivalently, if $\theta$ has local $\ddc$-potentials), then the form $ \theta+\ddc\varphi$ can be chosen to be K\"ahler on $U$.
\end{lemma}


\begin{proof} We follow ideas from the proof of \cite[Theorem 4]{Dem90}.
By the definition of a smooth function on a complex space, $\varphi_A$
extends to a smooth function on a neighborhood of $A$.  Replacing $\theta$
by $\theta+\ddc\varphi_A$, we may therefore assume that $\theta|_A$ is weakly K\"ahler.

Let $M:= \dim A+1$ and $A_0:=A$. Let $A_j$ be the union of irreducible components of $A_{j-1}$ of dimension $\le M-1-j$ and $\Sing(A_{j-1})$ for $1 \le j \le M-1$. We put $A_M:= \varnothing$.  
We obtain a stratification
\[
 A=A_0\supset A_1\supset\cdots\supset A_M=\varnothing,
 \qquad S_j:=A_j\setminus A_{j+1},
\]
whose stratum $S_j$ is either empty or a smooth manifold of pure dimension $M-1-j$.  We construct, by inverse induction over $0 \le j \le M$, a smooth function $\varphi_j$ on a
neighborhood $D_j$ of $A$, an open neighborhood $D'_j$ of $A_j$ (in $X$) such that the following three properties hold:


(i) The restriction $(\theta+\ddc \varphi_j)|_A$ is weakly Kähler on $A$, 

(ii) For every point $x\in D'_j$, there exist an local embedding $\rho_x: U_x \to W$ in $\C^N$ and a smooth form $\theta'$ on $W$, smooth function $\varphi'_j$ on $W$ such that $\theta'|_{U_x}= \theta$,  and $\varphi'_j|_{U_x}= \varphi_j$ and  $\theta'+ \ddc \varphi'_j$ is a Hermitian form on $W$. 

(iii) If $\theta$ is locally the restriction of a closed smooth form under local embeddings, then $\theta'$ in (ii) can be chosen to be closed.

Since $A_M$ is empty, it suffices to choose $\varphi_M:=0$, $D_M:= X$ and $D'_M= \varnothing$. 
Suppose now that for a $0\le j \le M-1$ what we want is true for  $A_{j+1}$, \emph{i.e.}, there are an open neighborhoods $D_{j+1},D'_{j+1}$ of $A$, $A_{j+1}$ respectively, and  a smooth function $\varphi_{j+1}$ on $D_{j+1}$ so that the above three properties hold for $\varphi_{j+1}$.  

Let $U_1 \Subset U:=D'_{j+1}$ be an open subset of $D'_{j+1}$ containing $A_{j+1}$.
Cover $A_j\backslash U$ by a finitely many open subsets $P_1, \ldots, P_m$ on $X$ 
such that $P_1,\ldots, P_{m}$ do not intersect $U_1$.
We note that $A_j\cap P_s \subset S_j$, which is smooth for $1 \le s \le m$.  By shrinking $P_s$ if necessary, we can assume that there is an embedding $\rho_s: P_s \hookrightarrow P_s'$, which is an open subset of $\C^{N_s}$, for $1\le s\le m$ and $A_{j} \cap P_s =S_{j}\cap P_s$ is given by $z'_s=0$ on $P'_s$, where $(z_{s,1},\ldots, z_{s,N_s})$ are the coordinates on $\C^{N_s}$ for $1\le s \le m$ and $z'_s=(z_{s,r_s+1}, \ldots , z_{s,N_s})$ ($r_s$ is the dimension of $A_j\cap P_s$). 

Let $\chi_s$ be  a smooth nonnegative cut-off function in $P'_s$  such that

(a) $\supp\chi_s \Subset P'_s$ for $1\le s\le m$,

(b) $\sum_{s=1}^m \chi_s \circ \rho_s(x) >0$ for every $x$ in an open neighborhood of  $A_j\backslash U$.

For $0<\epsilon\le 1/2$, consider on $P_s'$ the function
\[
        g_{s,\epsilon}
        :=
        \epsilon^3
        \log\bigl(1+\epsilon^{-4}|z_s'|^2\bigr),
        \qquad
        |z_s'|^2
        :=
        \sum_{\ell=r_s+1}^{N_s}|z_{s,\ell}|^2.
\]
This is the scaled logarithmic function used in
\cite[Theorem~4]{Dem90}. On every fixed relatively compact subset of
$P_s'$, it satisfies
\[
        0\le g_{s,\epsilon}
        \le C\epsilon^3|\log\epsilon|,
        \qquad
        |dg_{s,\epsilon}|\le C\epsilon,
        \qquad
        dd^c g_{s,\epsilon}\ge 0,
\]
where $C$ is independent of $\epsilon$. Indeed,
\[
        dg_{s,\epsilon}
        =
        \frac{\epsilon^{-1}d|z_s'|^2}
        {1+\epsilon^{-4}|z_s'|^2},
\]
and hence
\[
        |dg_{s,\epsilon}|
        \le
        C\frac{\epsilon^{-1}|z_s'|}
        {1+\epsilon^{-4}|z_s'|^2}
        \le C\epsilon.
\]
Moreover, along $z_s'=0$,
\[
        g_{s,\epsilon}=0,\qquad
        dg_{s,\epsilon}=0,\qquad
        dd^c g_{s,\epsilon}
        =
        \epsilon^{-1}dd^c|z_s'|^2.
\]
Define a smooth function on $P'_s$ by
\[
        h_{s,\epsilon}
        := \chi_s g_{s,\epsilon},
\]
Using  the above estimates for $g_{s,\epsilon},dg_{s,\epsilon}$, we obtain
\begin{align}\label{ine-chanduoiddchs}
\ddc h_{s,\epsilon} \ge \chi_s \ddc g_{s,\epsilon}- C\epsilon\omega_{\C^{N_s}} \ge -C\epsilon\omega_{\C^{N_s}},
\end{align}
for some constant $C>0$ independent of $\epsilon$,
and 
\begin{align}\label{ine-chanduoiddchs2}
\ddc h_{s,\epsilon} = \epsilon^{-1} \chi_s \ddc |z'_{s}|^2
\end{align}
at every point in $A_j \cap P_s$.
Define
$$\varphi_j:= \sum_{s=1}^{m} h_{s,\epsilon}\circ\rho_s+\varphi_{j+1} $$
which is a smooth function on an open neighborhood of $A$. By (\ref{ine-chanduoiddchs}) and the fact that $(\theta+\ddc \varphi_{j+1})|_A$ is Hermitian on $A$,   we see that the condition (i) is satisfied for $\varphi_j$ if $\epsilon$ is small enough.

Let $x\in U$. By the induction hypothesis, there is a local embedding $\rho: U_x \to W$ in $\C^N$ and $\varphi'_{j+1}$, $\theta'$ are smooth extension of $\varphi_{j+1}, \theta$ respectively in $W$ such that 
$$\ddc \varphi'_{j+1}+ \theta' \ge \delta\omega_{\C^N}$$
on $W$ for some constant $\delta>0$. Because of (\ref{ine-chanduoiddchs}), by decreasing $\epsilon$ so that $h_{s,\epsilon}\circ \rho_s$ admits a smooth extension $h'_s$ to $W$ so that 
$$\ddc h'_s \ge -\delta \omega_{\C^N}/(2m).$$ 
Consequently, for $\varphi'_j:= \varphi'_{j+1}+ \sum_s h'_s$, we obtain 
$$\theta'+ \ddc \varphi'_{j} \ge  \delta \omega_{\C^N}/2$$
on $W$, hence, it is a Hermitian form on $W$. 

Since $A_j \backslash U$ is compact, we can find a constant $c>0$ such that for every $x\in A_j \backslash U$, there exists $1\le s \le m$ such that 
\begin{align}\label{ine-chirho00}
\chi_s \circ \rho_s(x)\ge 2c.
\end{align}
Fix now such a point $x$ and assume that (\ref{ine-chirho}) holds for $s=1$. 
We use the local embedding $\rho=\rho_1: U_x\subset P_1 \to W=P'_1$ as above. By shrinking $U_x$ if necessary, using (\ref{ine-chirho00}) gives
\begin{align}\label{ine-chirho}
\chi_1 \circ \rho_1\ge c
\end{align}
on $U_x$.
Let $\theta', \varphi'_{j+1}$ be arbitrary smooth extensions of $\theta, \varphi_{j+1}$ to $W$.
By (\ref{ine-chanduoiddchs2}) for $s=1$, (\ref{ine-chirho}) and the fact that $\theta+\ddc \varphi_{j+1}$ is positive along $A_j\cap P_s$, we see that 
$$\theta'+ \ddc(\varphi'_{j+1}+ h_{1,\epsilon})$$
is Hermitian in $W$ at every point on $\rho(A_j\cap U_x)$  if $\epsilon$ is small enough. Hence by smoothness, the form $\theta'+ \ddc(\varphi'_{j+1}+ h_{1,\epsilon})$ is Hermitian on an open neighborhood of  $\rho(A_j\cap U_x)$ in $W$. 

  For $s$ such that $\supp \chi_{s}\circ\rho_s \cap U_x\not = \varnothing$, by shrinking $W$ if necessary, there is a holomorphic map $F: W \to P'_s$ with  $F\circ \rho= \rho_s$ on $U_x$.  Let $h'_s:= h_{s,\epsilon}\circ F$ which is a smooth extension of $h_{s,\epsilon}\circ \rho_s$.  Using (\ref{ine-chanduoiddchs}) and shrinking $P_1,P_s,W$ if necessary, one gets 
\begin{align}\label{ine-chanduokhphay}
\ddc h'_s \gtrsim - \epsilon \omega_{\C^N}.
\end{align}
This combined with the Hermitian property of $\theta'+ \ddc(\varphi'_{j+1}+ h_{1,\epsilon})$  show that 
$$\theta'+\ddc \varphi'_j$$
is Hermitian in a small open neighborhood of $\rho(U_x\cap A_j)$, where 
$\varphi'_j:= \varphi'_{j+1}+ \sum_s h'_s$. Thus (ii) and (iii) follow by choosing $\epsilon$ small enough (this can be done uniformly because there are only finitely many charts $P_s$'s).

If $\theta$ is locally the restriction of a closed smooth form under local embeddings, then the form $\theta'$ can be chosen to be $\ddc$-exact for every $s$. Thus $\theta'+ \ddc \varphi'_j$ is a K\"ahler form on $W$. The proof is complete.         
\end{proof}

Let $\delta \in (0,1/2]$, and 
$\chi_\delta\colon\mathbb R\to\mathbb R$ be a smooth convex function such
that
\[
 0\leq\chi_\delta'\leq1,
 \qquad
 \chi_\delta(t)=0\text{ for }t\leq-\delta,
 \qquad
 \chi_\delta(t)=t\text{ for }t\geq\delta.
\]

\begin{lemma} \label{le-noihaipsh} Let $Y$ be a complex space  of pure dimension, let $A$ be an analytic subset of $Y$. Let $\theta$ be a $(1,1)$-form on $Y$. Let $u$ be  a  smooth $\theta$-psh function on $Y\backslash A$. Let $v$ be a  smooth $\theta$-psh function on an open neighborhood $U$ of $A$ such that $v>u+2\delta$ near $A$ and  
$$u(x) \ge \limsup_{z \in U\to x}v(z)+2\delta$$
for $x\in \partial U$.  Define
\[
w=
\begin{cases}
v+\chi_\delta(u-v),&\text{on }U\setminus A,\\
u,&\text{on }Y\setminus U,
\end{cases}
\]
Then, $w$ extends to a  smooth $\theta$-psh function on $Y$.
\end{lemma}


\proof Since $w=v$ near $A$, we extend it to be equal to $v$ near $A$. The hypothesis implies that $w=u$ near $\partial U$.  One sees that $w$ is a smooth function on $Y$.  Therefore, $w$ is quasi-psh on $Y$.  We check that $w$ is $\theta$-psh. We already have this property near $A$ and on $Y\backslash U$ and near $\partial U$ (because $w=v$ near $A$, $w=u$ on $Y\backslash U$ and near $\partial U$).  On $U\backslash A$, since $u,v$ are smooth, we compute 
\begin{align*}
 \theta+\ddc w
 &=\chi_\delta'(u-v)(\theta+\ddc u)
   +(1-\chi_\delta'(u-v))(\theta+\ddc v)\\
 &\quad+\chi_\delta''(u-v)
      d(u-v)\wedge\dc(u-v)
\end{align*}
which is a positive current (to be rigorous, one should extend $u,v$ to smooth functions on local ambient spaces under local embeddings, then perform the above computations and restrict to $U$ after that). 
The proof is complete.
\endproof


Let
$\alpha\in\widetilde H_{\mathrm{BC}}^{1,1}(X,\mathbb R)$. We say that $\alpha$ is strongly nef if for every constant $\epsilon>0$, there exists a weakly Kähler form in $\alpha+ \epsilon \{\omega\}$. This is equivalent to saying that $\alpha$ is the limit of a sequence of weakly Kähler classes. We will see later that strongly nef classes are the same as nef ones. However, at this stage of the proof of Theorem \ref{thm:main-intro}, we have to consider this notion. Observe that for a strongly nef class $\alpha$ and every analytic subspace $A$ of $X$, we have $\alpha|_A$ is also strongly nef. 

\begin{proposition}
\label{prop:ambient-induction}
Let $X$ be a compact weakly K\"ahler space, and let
$\alpha\in\widetilde H_{\mathrm{BC}}^{1,1}(X,\mathbb R)$ be a strongly nef
class represented by a smooth closed real $(1,1)$-form $\theta$.  Assume that
\[
 \int_V\alpha^{\dim V}>0
 \tag{$*$}
\]
for every positive-dimensional irreducible analytic subset $V\subset X$.
Then $\alpha$ contains a weakly K\"ahler form.  If $\theta$ has smooth local
$\ddc$-potentials, then $\alpha$ contains a K\"ahler form.
\end{proposition}

We note that $X$ is not necessarily pure-dimensional.

\begin{proof}
We argue by induction on $n:=\dim X$, where this denotes the maximum of the
dimensions of the irreducible components.  The assertion is immediate for
$n=0$.  Fix a weakly K\"ahler form $\omega_X$ on $X$.  Write
$X_1,\ldots,X_s$ for the positive-dimensional irreducible components of $X$ and
choose strong projective resolutions
\[
 \pi_i:X_i'\longrightarrow X_i.
\]
By Lemma~\ref{le-blowup}, every $X_i'$ is K\"ahler; fix a K\"ahler form
$\omega_i$ on it.  Let
\[
 D:=\operatorname{Sing}(X) 
\]
Thus $D\cap X_i$ is a proper analytic subset of $X_i$.

For each $i$, the class $\pi_i^*(\alpha|_{X_i})$ is nef.  Indeed, pulling
back a nef approximation on $X$ and using
$\pi_i^*\omega_X\leq C_i\omega_i$ gives the required approximation on
$X_i'$.  Moreover, if $n_i:=\dim X_i$, then
\[
 \int_{X_i'}\pi_i^*(\alpha|_{X_i})^{n_i}
   =\int_{X_i}\alpha^{n_i}>0.
\]
Theorem~\ref{thm:current}, applied separately to every component, therefore
gives
\[
 \widetilde T_i=\pi_i^*\theta+dd^c\widetilde u_i
       \geq2\varepsilon_i\omega_i,
\]
where $\varepsilon_i>0$ and $\widetilde u_i$ is smooth away from a proper
analytic subset $B_i\subset X_i'$ and has analytic singularities there.

Set
\[
 Z_i:=B_i\cup\operatorname{Exc}(\pi_i)
          \cup\pi_i^{-1}(D\cap X_i).
\]
By Lemma~\ref{le-Demailly-exten.lemma}, we can add a sufficiently small perturbation by
a quasi-psh function with logarithmic poles along $Z_i$ to $\widetilde u_i$ allow us to retain
the estimate $\widetilde T_i\geq\varepsilon_i\omega_i$ and arrange that
\[
 \widetilde u_i\longrightarrow-\infty\quad\text{along }Z_i.
\]
Put
\[
 A:=\bigcup_i\pi_i(Z_i).
\]
Observe that $A$ contains $D$.  By Remmert's theorem, $A$ is analytic, and $A\cap X_i$ is proper in every
$X_i$; in particular, $\dim A<n$.  The complement $X\setminus A$ is the
disjoint union of the smooth sets $X_i\setminus A$ (and possibly isolated
points).  Hence, by putting $u:= \widetilde u_i \circ \pi_i^{-1}$ on $X_i \backslash A$, we obtain  a smooth
function $u$ on $X\setminus A$.  Since there are only finitely many
components and $\pi_i^*\omega_X\leq C_i\omega_i$, after decreasing a single
constant $\varepsilon>0$ we have
\[
 \theta+\ddc u\geq\varepsilon\omega_X
 \quad\text{on }X\setminus A.
\]
The logarithmic poles also give
\[
 u(x)\longrightarrow-\infty\quad\text{as }x\longrightarrow A
\]
along every positive-dimensional component (note that we don't claim that u is quasi-psh on $X$, and we will not need this).  If $A=\varnothing$, this
already proves the proposition, so assume that $A\ne\varnothing$.

Recall that the restriction $\alpha|_A$ is still strongly nef, and condition $(*)$ remains valid for
every positive-dimensional irreducible analytic subset of $A$.  Since
$\dim A<n$, the induction hypothesis shows that $\alpha|_A$ contains a
weakly K\"ahler form.  If $\theta$ is locally the restriction of a closed smooth form under local embeddings, the induction
hypothesis gives a K\"ahler form on $A$.  Lemma~\ref{le-extension-DP} then
yields a neighborhood $U$ of $A$ and a smooth function $v$ on $U$ such that
\[
 \theta+\ddc v\geq\varepsilon\omega_X
 \quad\text{on }U,
\]
after shrinking $U$ and decreasing $\varepsilon$.

Choose an open set $U_0$ with
\[
 A\subset U_0\Subset U.
\]
Adding a constant to $v$ does not change $\ddc v$; hence, after subtracting a
sufficiently large constant, we may arrange that
\[
 u-v>2\delta
\]
on a neighborhood of $\partial U_0$, for some $0<\delta<1/2$.  Let
$\chi_\delta\colon\mathbb R\to\mathbb R$ be a smooth convex function such
that
\[
 0\leq\chi_\delta'\leq1,
 \qquad
 \chi_\delta(t)=0\text{ for }t\leq-\delta,
 \qquad
 \chi_\delta(t)=t\text{ for }t\geq\delta.
\]
On $U_0\setminus A$ set
\[
 w:=v+\chi_\delta(u-v).
\]
Because $u\to-\infty$ along $A$, we have $w=v$ near $A$.  Because
$u-v>2\delta$ near $\partial U_0$, we have $w=u$ there.  Thus $w$ extends to
a global smooth function on $X$ by setting it equal to $v$ near $A$ and to
$u$ on $X\setminus U_0$. Lemma \ref{le-noihaipsh}, applied to irreducible components of $X$, tells us that $w$ is $\theta$-psh and 
\[
 \theta+\ddc w\geq\varepsilon\omega_X
\]
as currents on every irreducible component of $X$.  Near $A$ the inequality was obtained in local ambient
embeddings by Lemma~\ref{le-extension-DP}, while outside $A$ the space is
smooth.  Thus $\theta+\ddc w$ is Hermitian in the sense fixed in Section~1,
not merely positive as a current.  It is closed, and hence is a weakly
K\"ahler form in $\alpha$.
If $\theta$ is locally the restriction of a closed smooth form under local embeddings, then
$\theta+\ddc w$ is a Kähler form because $\theta+ \ddc v$ and $\theta+\ddc u$ are so locally on $U$ and $X\backslash A$ respectively. The proof is finished.
\end{proof}

We now apply the proposition to the relevant cones.  Let
\[
 \widetilde{\K}
 \subset\widetilde H_{\mathrm{BC}}^{1,1}(X,\mathbb R)
\]
be the weakly K\"ahler cone, \emph{i.e.}, the cone of weakly Kähler classes in $\widetilde H_{\mathrm{BC}}^{1,1}(X,\mathbb R)$. Let
$\widetilde{\Pcone}$ be the cone of classes $\alpha \in \widetilde H_{\mathrm{BC}}^{1,1}(X,\mathbb R)$  such that 
 $$\int_Y\alpha^p>0$$
 for every positive-dimensional irreducible  $ Y\subset X,\ p=\dim Y$.
Inside $H_{\mathrm{BC}}^{1,1}(X,\mathbb R)$, let $\K$ be the K\"ahler cone
and put
\[
 \Pcone:=\widetilde{\Pcone}\cap H_{\mathrm{BC}}^{1,1}(X,\mathbb R).
\]

\begin{proof}[Proof of Theorem \ref{thm:main-intro}]
We prove first the weakly K\"ahler statement. 
Observe
$\widetilde{\K}\subset\widetilde{\Pcone}$, and
$\widetilde{\K}$ is non-empty, open and convex (see Lemma \ref{le-convexKähler}).  It remains to prove that it is closed relative to $\widetilde{\Pcone}$.
Let
\[
 \alpha\in\overline{\widetilde{\K}}\cap\widetilde{\Pcone}.
\]
Being a limit of weakly K\"ahler classes, $\alpha$ is strongly nef. 
Since
$\alpha\in\widetilde{\Pcone}$, it satisfies the numerical hypothesis of
Proposition~\ref{prop:ambient-induction}.  That proposition shows that
$\alpha$ is weakly K\"ahler.  Therefore
\[
 \overline{\widetilde{\K}}\cap\widetilde{\Pcone}
 =\widetilde{\K}.
\]
Hence $\widetilde{\K}$ is both open and closed in
$\widetilde{\Pcone}$ and is therefore a union of connected components.  Its
convexity makes it connected, so it is exactly one component.

If $X$ is K\"ahler, then $\K$ is nonempty, open, and convex.  The final
assertion of Proposition~\ref{prop:ambient-induction} gives
$\overline{\K}\cap\Pcone=\K$, so the same argument shows that $\K$ is one
connected component of $\Pcone$.
\end{proof}

\begin{remark} \label{re-conePkhacrong} The proof of Theorem \ref{thm:main-intro} also proves that for every  compact weakly Kähler  space $X$, if $\mathcal{P} \not =\varnothing$, then $\mathcal{K} \not = \varnothing$ and $\mathcal{K}$ is a connected component of $\mathcal{P}$.
\end{remark}

We record the following useful mixed-intersection form.

\begin{corollary}\label{prop:mixed}
Let $X$ be a compact weakly K\"ahler space, 
fix a weakly K\"ahler class
$\omega$, and let
$\alpha\in\widetilde H_{\mathrm{BC}}^{1,1}(X,\mathbb R)$ and
assume that
\[
 \int_Y\alpha^k\wedge\omega^{p-k}>0,
 \qquad 1\leq k\leq p=\dim Y,
\]
for every positive-dimensional irreducible analytic subset $Y\subset X$.
Then $\alpha$ is weakly K\"ahler.  If $\alpha$ has smooth local
$\ddc$-potentials and $\omega$ is a K\"ahler class, then $\alpha$ is
K\"ahler.
\end{corollary}

\begin{proof}
For every $t\geq0$ and every irreducible $p$-dimensional analytic subset
$Y$, we have
\[
 \int_Y(\alpha+t\omega)^p
 =\sum_{k=0}^p\binom pk t^{p-k}
   \int_Y\alpha^k\wedge\omega^{p-k}>0.
\]
For \(t>0\), the \(k=0\) term is strictly positive, while for
\(t=0\), the \(k=p\) term is strictly positive. Hence the displayed
integral is positive for every \(t\geq0\).
Thus the entire ray $\alpha+t\omega$, $t\geq0$, lies in
$\widetilde{\Pcone}$.  For $t\gg1$, the class $\alpha+t\omega$ is weakly
K\"ahler.  Since the ray is connected, Theorem~\ref{thm:main-intro} implies that
it remains in the weakly K\"ahler component down to $t=0$.  Hence $\alpha$
is weakly K\"ahler.

If $\alpha$ has local potentials and $\omega$ is a K\"ahler class, then the
ray lies in $\Pcone$, and for
$t\gg1$ it consists of K\"ahler classes.  The K\"ahler part of
Theorem~\ref{thm:main-intro} therefore gives $\alpha\in\K$.
\end{proof}

\begin{lemma} \label{le-Smsoothpositiveascurrents} Let $X,\omega$ be as in Corollary \ref{prop:mixed}. Assume furthermore that $X$ is pure-dimensional. 
  Let $S$ be a smooth $d$-closed real $(1,1)$-form on $X$ which is
positive as a current.  Then, for every irreducible analytic subset
$Y\subset X$ of dimension $p$ and every $0\leq k\leq p$, one has
\begin{equation}\label{eq:smooth-positive-mixed}
  \int_Y S^k\wedge\omega^{p-k}\geq 0.
\end{equation}
\end{lemma}
 We recall that on a singular space, positivity in the sense of
currents cannot simply be replaced by pointwise positivity in an ambient
embedding. 

\begin{proof}
Choose an irreducible component $X_i$ containing $Y$ and a strong
projective resolution $\pi\colon X_i'\to X_i$. Since $S$ is smooth on $X$, its restriction to $X_i$ is also smooth. It follows that $\pi^* S$ is smooth on $X_i'$. On the other hand, by the semi-positivity of $S$ on $X_i \cap \Reg(X)$, we also have that
$\pi^*S$ is semipositive on $\pi^{-1}(X_i \cap \Reg(X))$, which is dense in $X_i'$. Thus $\pi^* S$ is semipositive on all of $X_i'$ by
continuity.  The form $\pi^*\omega$ is semipositive as well.  Let $F$ be an
irreducible component of $\pi^{-1}(Y)$ whose image under $\pi$ is equal to $Y$. Put
$r:=\dim F-p$, and choose a K\"ahler form $\Omega$ on $X_i'$; such a form
exists by Lemma~3.3.  The current
\[
  \pi_*\bigl([F]\wedge\Omega^r\bigr)
\]
is a positive closed current of bidimension $(p,p)$ supported on the
irreducible $p$-dimensional set $Y$.  The support theorem, applied to local embeddings of $X$, gives
\[
  \pi_*\bigl([F]\wedge\Omega^r\bigr)=c[Y]
\]
for a constant $c>0$; positivity of $c$ follows by integrating $\Omega^r$
on a generic fiber of $F\to Y$.  It follows that
\[
  c\int_Y S^k\wedge\omega^{p-k}
   =\int_F (\pi^*S)^k\wedge(\pi^*\omega)^{p-k}\wedge\Omega^r\geq 0,
\]
which proves \eqref{eq:smooth-positive-mixed}.
\end{proof}

\begin{corollary}\label{cor:03}
Let $X$ be a compact weakly K\"ahler space of pure dimension, and let
$\alpha\in\widetilde H_{\mathrm{BC}}^{1,1}(X,\mathbb R)$.  Then
$\alpha$ is nef if and only if there exists a weakly K\"ahler class $\omega$
such that
\[
 \int_Y\alpha^k\wedge\omega^{p-k}\geq0
\]
for every irreducible analytic subset $Y\subset X$, $p=\dim Y$, and every
$1\leq k\leq p$.
\end{corollary}

\begin{proof}
Assume first that $\alpha$ is nef, and fix a weakly K\"ahler class $\omega$.
For every $\varepsilon>0$, the class
$\alpha+\frac{\varepsilon}{2} \omega$ contains a smooth closed form which is positive as currents on $X$. Lemma \ref{le-Smsoothpositiveascurrents} implies
\[
 \int_Y(\alpha+\varepsilon\omega)^k\wedge\omega^{p-k}>0.
\]
Letting $\varepsilon\downarrow0$ gives the required nonnegative
inequalities.

Conversely, assume the inequalities.  For every $\varepsilon>0$ and
$1\leq k\leq p$,
\[
 \int_Y(\alpha+\varepsilon\omega)^k\wedge\omega^{p-k}
 =\sum_{j=0}^k\binom kj\varepsilon^{k-j}
   \int_Y\alpha^j\wedge\omega^{p-j}>0,
\]
because all terms are nonnegative and the $j=0$ term is strictly positive.
Corollary~\ref{prop:mixed} shows that $\alpha+\varepsilon\omega$ is
weakly K\"ahler for every $\varepsilon>0$.  Therefore, $\alpha$ is nef.
\end{proof}

\begin{corollary}\label{cor:04}
Let $X$ be a compact weakly K\"ahler space  of pure dimension, and let
$\alpha\in\widetilde H_{\mathrm{BC}}^{1,1}(X,\mathbb R)$.  Then
$\alpha$ is nef if and only if, for every irreducible analytic subset
$Y\subset X$ of dimension $p$ and every weakly K\"ahler class $\omega$,
\[
 \int_Y\alpha\wedge\omega^{p-1}\geq0.
\]
\end{corollary}

\begin{proof}
Necessity follows by applying positivity to
$\alpha+\varepsilon\omega$, Lemma \ref{le-Smsoothpositiveascurrents} and letting $\varepsilon\downarrow0$.

For sufficiency, we follow the polynomial argument of
\cite[Theorem~4.4]{DP04}. By the latter reference,  for each $p\geq1$ there is a polynomial
$A_p(t,\delta)$, depending continuously on $\delta$ and of degree at most
$p-1$ in $t$, such that
\begin{equation}\label{eq:poly}
 (y-\delta x)^p-(1-\delta)^p x^p
 =(y-x)\int_0^1 A_p(t,\delta)
       ((1-t)x+ty)^{p-1}\,dt.
\end{equation}
Moreover, there is
$\delta_0\in(0,1)$ such that
\[
 A_p(t,\delta)>0
\]
for $0\leq t\leq1$, $0\leq\delta\leq\delta_0$, and all  $p\le n$.

Choose a sufficiently large positive multiple $\omega$ of a fixed weakly
K\"ahler class so that
$\omega':=\alpha+\omega$ is weakly K\"ahler.  Substituting
$x=\omega$ and $y=\omega'$ into \eqref{eq:poly}, we obtain, for every
irreducible $p$-dimensional analytic subset $Y$,
\begin{align*}
 &\int_Y(\alpha+(1-\delta)\omega)^p
 -(1-\delta)^p\int_Y\omega^p\\
 &\quad=\int_0^1 A_p(t,\delta)
   \left(\int_Y\alpha\wedge
   ((1-t)\omega+t\omega')^{p-1}\right)dt\geq0.
\end{align*}
The inequality follows from the hypothesis, since
$(1-t)\omega+t\omega'$ is weakly K\"ahler.  Hence
$\alpha+(1-\delta)\omega\in\widetilde{\Pcone}$ for
$0\leq\delta\leq\delta_0$.  This segment starts at the weakly K\"ahler class
$\alpha+\omega$, so Theorem~\ref{thm:main-intro} implies that
\[
 \alpha+r\omega\in\widetilde{\K},
 \qquad r:=1-\delta_0<1.
\]
Repeating the same argument with $r^\nu\omega$ in place of $\omega$ gives
\[
 \alpha+r^\nu\omega\in\widetilde{\K}
 \qquad(\nu=0,1,2,\ldots).
\]
These classes converge to $\alpha$, and therefore $\alpha$ is nef.
\end{proof}

To finish this section, we present now proofs of Corollary \ref{cor-main-equalitycone} and Theorem \ref{thm-projective}.

\begin{proof}[Proof of Corollary \ref{cor-main-equalitycone}] 
Fix a weakly K\"ahler form $\omega$ on $X$. 
The inclusion $\widetilde{\mathcal K}\subset
\widetilde{\mathcal K}^{\prime\prime}$ was already noted.  Conversely, let
$\alpha\in\widetilde{\mathcal K}^{\prime\prime}$ and choose a smooth
representative $T$ of $\alpha$ which is a K\"ahler current.  After decreasing
the positivity constant, we may write
\[
  T\geq\delta\omega
\]
for some $\delta>0$.  Put
\[
  S:=T-\delta\omega\geq 0,
  \qquad
  \beta:=\alpha-\delta\{\omega\}=\{S\}.
\]
By Lemma \ref{le-Smsoothpositiveascurrents}, we have
\[
  \int_Y\beta^k\wedge\omega^{p-k}\geq 0
\]
for every irreducible $p$-dimensional analytic subset $Y\subset X$ and every
$1\leq k\leq p$.  Corollary \ref{cor:03} shows that $\beta$ is nef.  Applying the same
corollary to $\beta$ and expanding the binomial, we obtain
\[
\begin{split}
  \int_Y\alpha^k\wedge\omega^{p-k}
   &=\sum_{j=0}^k\binom{k}{j}\delta^{k-j}
       \int_Y\beta^j\wedge\omega^{p-j} \\
   &>0.
\end{split}
\]
Here every summand is nonnegative, while the term with $j=0$ equals
$\delta^k\int_Y\omega^p>0$.  Corollary \ref{prop:mixed} therefore implies that $\alpha$
is weakly K\"ahler.  Consequently,
\[
  \widetilde{\mathcal K}
   =\widetilde{\mathcal K}^{\prime\prime}.
\]

We next have the evident inclusions
\[
  \mathcal K
  \subset
  \widetilde{\mathcal K}'\cap H^{1,1}_{\mathrm{BC}}(X,\mathbb R)
  \subset
  \widetilde{\mathcal K}^{\prime\prime}
       \cap H^{1,1}_{\mathrm{BC}}(X,\mathbb R).
\]
Let $\alpha$ belong to the last cone.  By the equality just proved,
$\alpha$ is weakly K\"ahler.  Hence it is nef and
$\int_Y\alpha^{\dim Y}>0$ for every positive-dimensional irreducible
analytic subset $Y\subset X$.  Since $\alpha$ belongs to Bott--Chern
cohomology, it has a smooth representative with local $dd^c$-potentials.
The last assertion of Proposition \ref{prop:ambient-induction} then shows that $\alpha$ contains a
K\"ahler form.  Thus
\[
  \mathcal K
  =\widetilde{\mathcal K}'\cap H^{1,1}_{\mathrm{BC}}(X,\mathbb R)
  =\widetilde{\mathcal K}^{\prime\prime}
       \cap H^{1,1}_{\mathrm{BC}}(X,\mathbb R).
\]
Finally, let $\alpha=\{\theta\}$ be nef.  For every $\varepsilon>0$ there is
$f_\varepsilon\in C^\infty(X,\mathbb R)$ such that
\[
  \theta+dd^cf_\varepsilon\geq-\varepsilon\omega.
\]
It follows that
\[
  \theta+dd^cf_\varepsilon+2\varepsilon\omega
     \geq\varepsilon\omega
\]
is a smooth K\"ahler current in
$\alpha+2\varepsilon\{\omega\}$.  The first equality above shows that these
classes are weakly K\"ahler and converge to $\alpha$.  If $X$ is K\"ahler,
choose $\omega$ K\"ahler.  For $\alpha\in H^{1,1}_{\mathrm{BC}}(X,\mathbb R)$
the same classes belong to Bott--Chern cohomology, and the second equality
shows that they are K\"ahler.  This proves both approximation assertions.
\end{proof}

\begin{proof}[Proof of Theorem \ref{thm-projective}]
Let $A$ be a very ample line bundle on $X$, and let
$\omega=c_1(A,h)$ be the K\"ahler form obtained by restricting a
Fubini--Study metric.  By Chow's theorem, every irreducible analytic subset
of $X$ is algebraic.  Let $Y\subset X$ be irreducible of dimension $p$, fix
$1\leq k\leq p$, and choose general members
$H_1,\ldots,H_{p-k}\in|A|$ which meet $Y$ properly.  Write the resulting
effective cycle as
\[
  Y\cdot H_1\cdots H_{p-k}=\sum_{\ell}m_\ell Z_\ell,
  \qquad m_\ell>0,
\]
where every $Z_\ell$ is irreducible of dimension $k$.  For every smooth
$d$-closed real $(1,1)$-form representing a class $\alpha\in
\widetilde H^{1,1}_{\partial\bar\partial}(X,\mathbb R)$, the projection
formula gives
\begin{equation}\label{eq:hyperplane-mixed}
  \int_Y\alpha^k\wedge\omega^{p-k}
     =\sum_\ell m_\ell\int_{Z_\ell}\alpha^k.
\end{equation}
The identity may equivalently be checked after pulling everything back to a
resolution of $Y$; hence it is valid without any smoothness or normality
assumption on $X$.

Suppose first that $\alpha\in\widetilde{\mathcal P}$.  Every term on the
right-hand side of \eqref{eq:hyperplane-mixed} is strictly positive, and the
cycle on that side is nonzero.  Therefore
\[
  \int_Y\alpha^k\wedge\omega^{p-k}>0
  \qquad (1\leq k\leq p).
\]
Corollary~3.8 shows that $\alpha$ is weakly K\"ahler.  The reverse inclusion
is immediate, and hence
\[
  \widetilde{\mathcal K}=\widetilde{\mathcal P}.
\]
If in addition $\alpha\in H^{1,1}_{\mathrm{BC}}(X,\mathbb R)$, then
$\alpha$ has smooth local $dd^c$-potentials and $\omega$ is K\"ahler, so the
second assertion of Corollary \ref{prop:mixed} gives $\alpha\in\mathcal K$.  Thus
\[
  \mathcal K=\mathcal P,
\]
which also proves assertion~\textup{(i)}.

It remains to prove assertion~\textup{(ii)}.  If $\alpha$ is nef, the
necessity part of Corollary \ref{cor:03}, with $k=p$, gives
\[
  \int_Y\alpha^p\geq0
\]
for every positive-dimensional irreducible analytic subset $Y\subset X$.
Conversely, suppose that all these top self-intersections are nonnegative.
Formula \eqref{eq:hyperplane-mixed} then gives
\[
  \int_Y\alpha^k\wedge\omega^{p-k}\geq0
  \qquad (1\leq k\leq p)
\]
for every irreducible $p$-dimensional $Y\subset X$.  Corollary \ref{cor:03} implies
that $\alpha$ is nef, as desired.    
\end{proof}


\section{Non--K\"ahler loci} \label{sec-nonkahlerlocus}

Throughout this section, $X$ is a compact reduced K\"ahler space of pure
dimension and every class $\alpha$ under consideration belongs to the usual
Bott--Chern group
\[
    H_{\mathrm{BC}}^{1,1}(X,\mathbb R).
\]
The extended Bott--Chern group is not used in this section.  Additional
smoothness or normality assumptions on $X$ will be stated explicitly.
For a closed positive $(1,1)$-current $T$ on a reduced pure-dimensional
complex space, the Lelong number $\nu(T,x)$ is defined as follows.  If
$\rho:U\hookrightarrow W$ is a local embedding of an open neighborhood $U$
of $x$, then
\[
    \nu(T,x):=\nu(\rho_*T,\rho(x)).
\]
This definition is independent of the chosen local embedding.  The sets
$E_c(T)$ below are analytic by Siu's theorem, applied after pushforward in a
local embedding.
Put
\[
 E_c(T):=\{x\in X:\nu(T,x)\geq c\},
 \qquad c>0,
\]
and
\[
 \Eplus(T)
 :=\bigcup_{c>0}E_c(T)
 =\{x\in X:\nu(T,x)>0\}.
\]
The convention $E_c(T)=\{\nu(T,\cdot)>c\}$ used in
\cite[Definition~3.1]{DH24} gives the same union $\Eplus(T)$ when $c$ ranges
over all positive numbers.

\begin{definition}[Non--K\"ahler loci]\label{def:nk}
Assume that $\alpha\in H_{\mathrm{BC}}^{1,1}(X,\mathbb R)$ is big.
Following \cite{DH24,HPaun24}, define
the \emph{non--K\"ahler locus}
\[
 \EnK(\alpha)
 :=
 \bigcap_{\substack{T\in\alpha\\ T\text{ a K\"ahler current}\\
                     T\text{ has local potentials}}}
 \Eplus(T).
\]
\end{definition}
In the smooth setting, this definition is due to Boucksom \cite{Bou04}.  The
local-potential requirement in the normal analytic-space formulation is
explicit in \cite[Definitions~2.2 and~3.1]{DH24}. Let \(T=\theta+dd^c\varphi\in\alpha\).  Following \cite[Definition~4.11]{HPaun24} (also \cite[Definition~3.3]{DH24} and \cite[Definition~2.9]{HLR26}), we  say that \(T\) has \emph{weak
analytic singularities} if there exist a proper holomorphic map
\[
\pi:\widehat X\longrightarrow X
\]
from a normal complex space and a closed positive current
\[
\widehat T=\pi^*\theta+dd^c\psi\in\pi^*\alpha
\]
such that $\pi$ is bimeromorphic,  \(\psi\) has admissible singularities and
\(\pi_*\widehat T=T\).

\begin{definition}[Restricted non--K\"ahler locus]\label{def:restricted-nk}
Following \cite[Definition~4.18]{HPaun24}, assume that
$\alpha\in H_{\mathrm{BC}}^{1,1}(X,\mathbb R)$ is nef and big.  Its
\emph{restricted non--K\"ahler locus} is
\[
 \EnKas(\alpha)
 :=
 \bigcap_{\substack{T\in\alpha\\
                     T\text{ a K\"ahler current}\\
                     T\text{ has weak analytic singularities}}}
 \Eplus(T).
\]
\end{definition}

We recall the following observation.

\begin{proposition}\label{prop:comparison}
Let $X$ be a normal compact K\"ahler space and let
$\alpha\in H_{\mathrm{BC}}^{1,1}(X,\mathbb R)$ be nef and big.
Then
\[
 \EnK(\alpha)
 \subseteq
 \EnKas(\alpha)
 \subseteq
 \EnK(\alpha)\cup\Sing(X).
\]
Moreover, there exists a K\"ahler current $T\in\alpha$ with weak analytic
singularities such that
\[
 \Eplus(T)=\EnKas(\alpha).
\]
If $X$ is smooth, then
\[
 \EnKas(\alpha)=\EnK(\alpha).
\]
\end{proposition}

\begin{proof}
See \cite[Remark~4.19 and Corollary~4.20]{HPaun24}.  In the smooth case,
Demailly's regularization \cite{Dem92} provides K\"ahler currents with
analytic singularities without changing the ordinary non--K\"ahler locus.
\end{proof}

\begin{definition}[Null locus]\label{def:null}
Let $X$ be a reduced compact K\"ahler space of pure dimension and let
$\alpha\in H_{\mathrm{BC}}^{1,1}(X,\mathbb R)$ be nef and big.  Set
\[
 \Null(\alpha)
 :=
 \bigcup_{\substack{V\subset X\text{ irreducible analytic}\\
                     \dim V>0,\;\int_V\alpha^{\dim V}=0}}
 V.
\]
For $k=\dim V$, the intersection number can be defined on a resolution
$\mu:\widetilde V\to V$ by
\[
 \int_V\alpha^k
 :=
 \int_{\widetilde V}(\mu^*\alpha)^k.
\]
Equivalently, for a smooth local-potential representative, one may integrate
over $V_{\mathrm{reg}}$.  See \cite[Definition~4.9]{HPaun24}.
\end{definition}

\begin{theorem}[Collins--Tosatti; Hacon--P\u{a}un]\label{thm:null}
Let $\alpha\in H_{\mathrm{BC}}^{1,1}(X,\mathbb R)$ be nef and big.
\begin{enumerate}[label=\textup{(\roman*)}]
\item If $X$ is a compact K\"ahler manifold, then
\[
 \EnK(\alpha)=\Null(\alpha).
\]
\item If $X$ is a normal compact K\"ahler space, then
\[
 \EnKas(\alpha)=\Null(\alpha).
\]
In particular, $\Null(\alpha)$ is analytic.
\end{enumerate}
\end{theorem}

\begin{proof}
Part (i) is \cite[Theorem~1.1]{CT15}.  Part (ii) is
\cite[Theorem~4.21]{HPaun24}; see also \cite[Theorem~3.6]{DH24}.
\end{proof}

\begin{lemma}[Normalization of the null locus]\label{lem:normalization-reduction}
Let $X$ be a reduced compact K\"ahler space of pure dimension, let
$\nu:X^\nu\to X$ be its normalization, and let
$\alpha\in H_{\mathrm{BC}}^{1,1}(X,\mathbb R)$ be nef and big.
Then $\nu^*\alpha$ is nef and big and
\begin{equation}\label{eq:null-under-normalization}
    \nu\bigl(\Null(\nu^*\alpha)\bigr)=\Null(\alpha).
\end{equation}
In particular, $\Null(\alpha)$ is analytic.
\end{lemma}

\begin{proof}
The argument is componentwise, so we may work on one irreducible component.
The normalization is a finite projective modification of $X$; the relative
Fubini--Study argument of Lemma~\ref{le-blowup} therefore shows that $X^\nu$
is a compact K\"ahler space.  The class $\nu^*\alpha$ is nef: pull back the
smooth nef approximants on $X$ and compare $\nu^*\omega$ with a K\"ahler
form on $X^\nu$.

We verify bigness without pushing a current through a map with
positive-dimensional fibers.  Choose $S\in\alpha$ with
$S\geq\delta\omega$, take a projective resolution
$\mu:Y\to X^\nu$, and put $f:=\nu\circ\mu$.  Since $\alpha$ is an ordinary
Bott--Chern class, the pullback
\[
    f^*S=f^*\theta+\ddc(\varphi\circ f),
    \qquad S=\theta+\ddc\varphi,
\]
is well defined, and the local-potential pullback preserves
$f^*S\geq\delta f^*\omega$.  The class $f^*\{\omega\}$ is nef and has positive top
self-intersection by the projection formula.  Theorem~\ref{thm:current}
gives a K\"ahler current $K_\omega\in f^*\{\omega\}$.  Consequently,
\[
    f^*S-\delta f^*\omega+\delta K_\omega
\]
is a K\"ahler current in $f^*\alpha$.  Thus the nef class $f^*\alpha$ is
big.  The standard identity between the volume and the top
self-intersection of a nef class \cite{Bou04,BEGZ10} therefore gives
$\int_Y(f^*\alpha)^n>0$.  The projection formula
gives
\[
   \int_{X^\nu}(\nu^*\alpha)^n
     =\int_Y(f^*\alpha)^n>0,
   \qquad n=\dim X.
\]
The normal-space bigness criterion for nef classes
\cite[Lemma~2.35]{DHP24} now shows that $\nu^*\alpha$ is big.  This argument
applies separately to every component of $X^\nu$.

By Theorem~\ref{thm:null}\textup{(ii)},
$\Null(\nu^*\alpha)$ is analytic on each component of $X^\nu$.
If $W\subset X^\nu$ is irreducible of dimension $k>0$, put
$V:=\nu(W)$ and let $d_W$ be the generic degree of $\nu|_W$.  The projection
formula gives
\[
    \int_W(\nu^*\alpha)^k=d_W\int_V\alpha^k.
\]
Since $\nu$ is finite, every positive-dimensional irreducible analytic subset
upstairs and downstairs occurs in this way.  The displayed projection
formula proves \eqref{eq:null-under-normalization}.  Remmert's proper mapping
theorem then shows that $\Null(\alpha)$ is analytic.
\end{proof}

\begin{proposition} \label{pro-NullnamtrongEnK} Let $X$ be a compact Kähler space of pure dimension. Let $\alpha$ be a big and nef Bott-Chern class on $X$. Then  we have
\[
    \Null(\alpha)\subseteq\EnK(\alpha).
\]
\end{proposition}

\begin{proof}
Let $V$ be an irreducible analytic subset occurring in the definition of
$\Null(\alpha)$, and put $k:=\dim V>0$.  Thus
\begin{equation}\label{eq:null-V}
    \int_V\alpha^k=0.
\end{equation}
Suppose, for a contradiction, that
$V\not\subseteq\EnK(\alpha)$.  By the definition of
$\EnK(\alpha)$ in Definition~\ref{def:nk}, there is a K\"ahler current
$R\in\alpha$ such that
\[
    V\not\subseteq\Eplus(R).
\]
Since $\alpha$ is an ordinary Bott--Chern class, $R$ has local psh
potentials.
Fix a K\"ahler form $\omega$ on $X$.  Since $R$ is a K\"ahler current, after
decreasing the constant if necessary there is a number $\delta>0$ such that
\begin{equation}\label{eq:kahler-current-lower-bound}
    R\geq\delta\omega.
\end{equation}

Choose a strong projective desingularization
\(
    \pi\colon X'\to X
\)
and a K\"ahler form $\Omega$ on $X'$; if $X$ is reducible, we work on a
component of $X'$ whose image contains $V$.  Let $F$ be an irreducible
component of $\pi^{-1}(V)$ which dominates $V$, and put
\[
    r:=\dim F-k.
\]
As in the proof of Lemma \ref{le-Smsoothpositiveascurrents}, we get
\begin{equation}\label{eq:push-F-to-V}
    \pi_*\bigl([F]\wedge\Omega^r\bigr)=c[V]
\end{equation}
for some $c>0$.  Indeed, $c$ is the integral of $\Omega^r$ on a generic
fiber of $F\to V$.  Notice that only the positivity of $c$, and no estimate
uniform in $V$, is needed below.

We first pull $R$ back to $X'$.  Write
\[
    R=\theta+\ddc u,
\]
where $\theta$ is a smooth local-potential representative of $\alpha$ and
$u$ is quasi-plurisubharmonic on $X$, and set
\begin{equation}\label{eq:def-resolution-pullback-R}
    R':=\pi^*\theta+\ddc(u\circ\pi).
\end{equation}
This is a well-defined closed $(1,1)$-current in the class $\pi^*\alpha$.
Since $R-\delta\omega$ is a positive current with local psh potentials, its
pullback by $\pi$ is well defined and positive.  Therefore
\begin{equation}\label{eq:resolution-lower-bound}
    R'\geq\delta\pi^*\omega.
\end{equation}

We next verify the point in the argument for which the hypothesis
$V\not\subset\Eplus(R)$ is essential:
\begin{equation}\label{eq:generic-lelong-pullback-zero}
    \nu(R',F)=0.
\end{equation}
Here and below $\nu(\,\cdot\,,F)$ denotes the generic Lelong number along
$F$.  Choose a very general point $x\in V$ such that
$\nu(R,x)=0$, and a general $y\in F$ with $\pi(y)=x$.  Locally near $x$
write $\theta=\ddc h$ and put $\varphi:=u+h$, so that
$R=\ddc\varphi$ as a current; after adding a constant, assume
$\varphi\leq0$.  Choose a finite local projection
\[
    q\colon U\longrightarrow B\subset\mathbb C^n,
    \qquad q^{-1}(q(x))=\{x\}.
\]
The upper-semicontinuous representative of the trace $v:=q_*\varphi$ is a
psh potential of the positive current $q_*R$ on $B$; indeed,
$\ddc v=q_*R\geq0$.  Since $\varphi\leq0$ and the trace is the sum over the
finite fiber, counting multiplicities, one has
\begin{equation}\label{eq:trace-potential-comparison}
    \varphi\geq v\circ q
\end{equation}
on $U$ (the inequality first holds off the discriminant and then extends
by upper semicontinuity).  Demailly's direct-image estimate
\cite[Theorem~6]{Dem82Direct} yields
\[
    \nu(q_*R,q(x))=0.
\]
The map $q\circ\pi$ has maximal generic rank.  By \cite[Theorem~2 and Corollary~4]{Fav99} (or \cite{Kis00}), together with
\eqref{eq:trace-potential-comparison}, gives
\[
\begin{split}
    0\leq\nu(R',y)
      &=\nu(\varphi\circ\pi,y)\\
      &\leq\nu\bigl(v\circ q\circ\pi,y\bigr)
       \leq C_y\nu(q_*R,q(x))=0.
\end{split}
\]
This proves \eqref{eq:generic-lelong-pullback-zero}.  The same argument is
applied branch by branch at a reducible point.

Set
\[
    S:=R'-\delta\pi^*\omega\geq0.
\]
Subtracting the smooth form $\delta\pi^*\omega$ does not change Lelong
numbers, so \eqref{eq:generic-lelong-pullback-zero} gives $\nu(S,F)=0$.
By Demailly's regularization theorem \cite{Dem92}, there are currents
$S_j\in\{S\}$ with analytic singularities and numbers
$\varepsilon_j\downarrow0$ such that
\begin{equation}\label{eq:regularized-S}
    S_j\geq-\varepsilon_j\Omega,
    \qquad
    \nu(S_j,z)\leq\nu(S,z)\quad(z\in X').
\end{equation}
Define
\[
    T_j:=S_j+\delta\pi^*\omega+\varepsilon_j\Omega.
\]
Thus $T_j$ is a closed positive current with analytic singularities,
\begin{equation}\label{eq:Tj-class-and-bound}
    T_j\geq\delta\pi^*\omega,
    \qquad
    \{T_j\}=\pi^*\alpha+\varepsilon_j\{\Omega\}.
\end{equation}
Equations \eqref{eq:generic-lelong-pullback-zero} and
\eqref{eq:regularized-S} show that the singular set of $T_j$ does not
contain $F$.

Let $\mu\colon\widehat F\to F$ be a projective resolution, let
$g\colon\widehat F\to X'$ be the induced map, and set
$f:=\pi\circ g$.  Since $F$ is not contained in the singular set of
$T_j$, the pullback
\[
    \widehat T_j:=g^*T_j
\]
is well defined.  It is a closed positive current with analytic
singularities on the compact K\"ahler manifold $\widehat F$.  Moreover,
by \eqref{eq:Tj-class-and-bound},
\begin{equation}\label{eq:hat-Tj}
    \widehat T_j\geq\delta f^*\omega,
    \qquad
    \{\widehat T_j\}
      =f^*\alpha+\varepsilon_j g^*\{\Omega\}.
\end{equation}

We now apply the non-pluripolar mass argument from
\cite[Theorem~2.5]{CT15}.  Multilinearity and positivity of
non-pluripolar products \cite[Proposition~1.4]{BEGZ10}, together with
\eqref{eq:hat-Tj}, give
\begin{align}
 \int_{\widehat F}
    \big\langle\widehat T_j^k\big\rangle\wedge(g^*\Omega)^r
 &\geq
   \delta^k\int_{\widehat F}(f^*\omega)^k\wedge(g^*\Omega)^r \notag\\
 &=\delta^k c\int_V\omega^k>0.                 \label{eq:uniform-lower-NP}
\end{align}
The lower bound is independent of $j$; the second equality follows from
\eqref{eq:push-F-to-V}.

On the other hand, put
\[
    \beta_j:=f^*\alpha+\varepsilon_jg^*\{\Omega\}.
\]
This class is nef.  The cohomological upper bound for non-pluripolar
products \cite[Proposition~1.20]{BEGZ10} bounds the class of
$\langle\widehat T_j^k\rangle$ by the positive product
$\langle\beta_j^k\rangle$.  Here $g^*\{\Omega\}$ is nef and big, since
\[
  \int_{\widehat F}(g^*\Omega)^{k+r}=\int_F\Omega^{k+r}>0,
\]
and hence $\beta_j$ is nef and big; because $\beta_j$ is nef, its positive product is the
ordinary cup product $\beta_j^k$.  Pairing the resulting pseudoeffective
class inequality with the nef class $(g^*\{\Omega\})^r$ gives the required
upper bound.  More generally, the monotonicity for non-pluripolar products (see \cite{DDL18,WN19}, also
\cite{Vu21Relative}) gives
\begin{align}
 \int_{\widehat F}
    \big\langle\widehat T_j^k\big\rangle\wedge(g^*\Omega)^r
 &\leq
 \int_{\widehat F}
   \bigl(f^*\alpha+\varepsilon_jg^*\Omega\bigr)^k
       \wedge(g^*\Omega)^r                                      \notag\\
 &=\int_F
   \bigl(\pi^*\alpha+\varepsilon_j\Omega\bigr)^k\wedge\Omega^r.
                                                               \label{eq:upper-NP}
\end{align}
The last expression is a polynomial in $\varepsilon_j$.  Its constant
term is, by \eqref{eq:push-F-to-V} and \eqref{eq:null-V},
\[
    \int_F(\pi^*\alpha)^k\wedge\Omega^r
      =c\int_V\alpha^k=0.
\]
Consequently, the right-hand side of \eqref{eq:upper-NP} tends to zero as
$j\to\infty$, contradicting the uniform positive lower bound
\eqref{eq:uniform-lower-NP}.  This contradiction proves
$V\subseteq\EnK(\alpha)$, and hence
$\Null(\alpha)\subseteq\EnK(\alpha)$.  
\end{proof}

\begin{proof}[Proof of Theorem~\ref{thm:restricted-non-kahler-intro}]
We already know that 
\[
    \EnK(\alpha)
      \subset\EnKas(\alpha),
\]
Proposition \ref{pro-NullnamtrongEnK} tells us that $\Null(\alpha)\subset \EnK(\alpha)$, whereas by the normality of $X$ and Theorem \ref{thm:null}, we get 
\[
    \EnKas(\alpha)=\Null(\alpha).
\]
The desired equalities thus follow. 
\end{proof}

\bibliographystyle{alpha}
\bibliography{DP-refer-citations-revised}

\end{document}